\documentclass[12pt]{article}
\usepackage{amsmath, amsfonts, amsthm}
\usepackage{amssymb,mathrsfs, verbatim}
\usepackage{graphicx,amstext}
\usepackage[mathcal]{euscript}
\usepackage[latin1]{inputenc}
\newcommand{\ol}[1]{{\overline{#1}}}

\newcommand{\R}{\mathbb R}

\newcommand{\N}{\mathbb N}
\newcommand{\C}{\mathbb C}

\newcommand{\la}{\lambda}
\newcommand{\al}{\alpha}
\newcommand{\bt}{\beta}
\newcommand{\ve}{\varepsilon}
\newcommand{\vp}{\varphi}

\newcommand{\divx}{{\rm div}_{\mathbf{x}}}

\newcommand{\nablax}{{\nabla}_{\mathbf{x}}}

\newcommand{\deltax}{{\Delta}_{\mathbf{x}}}

\newtheorem{theorem}{Theorem}[section]
\newtheorem{proposition}[theorem]{Proposition}
\newtheorem{remark}[theorem]{Remark}
\newtheorem{lemma}[theorem]{Lemma}
\newtheorem{corollary}[theorem]{Corollary}
\newtheorem{definition}[theorem]{Definition}

\begin{document}
\title{The global solvability of initial-boundary value problem for nondiagonal
parabolic systems}

\author{Wladimir Neves$^1$, Mikhail Vishnevskii $^2$}

\date{}

\maketitle

\footnotetext[1]{ Instituto de Matem\'atica, Universidade Federal
do Rio de Janeiro, C.P. 68530, Cidade Universit\'aria 21945-970,
Rio de Janeiro, Brazil. E-mail: {\sl wladimir@im.ufrj.br.}}

\footnotetext[2]{Universidade Estadual do Norte Fluminense,
Alberto Lamego 2000, 28013-602, Rio de Janeiro, Brazil, and
Sobolev Institute of Mathematics, Koptuy prosp. 4630090,
Novosibirsky, Russia. E-mail: {\sl mikhail@uenf.br.}
%



\textit{Key words and phrases. Nondiagonal Parabolic systems,
Global in time solvability, Uniqueness and existence theorem,
Classical solutions.}}

\begin{abstract}
In this paper we study the quasilinear nondiagonal parabolic
type systems. We assume that the principal elliptic operator,
which is part of the parabolic system, has a divergence structure.
Under certain conditions it is proved the well-posedness of
classical solutions, which exist globally in time.
\end{abstract}

\maketitle

\section{Introduction}
\label{s.1}

The main purpose of this paper is to present some techniques and
results concerning global existence of classical solutions for
nondiagonal parabolic systems. To be precise, let
$(t,\mathbf{x})\in \mathbb{R}\times \mathbb{R}^{d}$, ($d \in \N$
fixed), be the points in the time-space domain. Throughout this
paper $\Omega \subset \R^d$ is an open bounded domain of class
$C^1$, $\mathbf{n}= (n^1,\ldots, n^d)$ is the unitary normal
vector field on $\partial \Omega =:\Gamma$.

\bigskip
For $T>0$ and $N \in \N$, we define $Q_T:= (0,T) \times \Omega$
and consider the vector function $\mathbf{u}: \ol{Q_T} \to \R^N$,
which is supposed to be governed by the following
reaction-diffusion system
\begin{equation}
\label{PS}
\begin{aligned}
   \partial_t u_\al(t,\mathbf{x}) +
   \divx \mathbf{f}_\al(\mathbf{x},\mathbf{u})=
   g_\al(\mathbf{x},\mathbf{u}),
   \quad (t,\mathbf{x}) \in Q_T,
\end{aligned}
\end{equation}
where $\mathbf{f}_\al$ is a given flux defined by
$$
  f^j_\al(\mathbf{x},\mathbf{u}):= \vp^j_\al(\mathbf{x},\mathbf{u})
  - A^{jk}_{\al\bt}(\mathbf{x},\mathbf{u}) \frac{\partial u_\bt}{\partial
  x_k}, \qquad (\al,\bt= 1,\ldots,N).
$$
Hereafter, the usual summation convention is used. Moreover, Greek
and Latin indices ranges respectively from 1 to $N$ and from 1 to
$d$. Although, we are not going to enter in physical details, we
should mention that there are many physical applications of the
above reaction-diffusion system, we list for instance: Flows in
porous media, diffusion of polymers, population dynamics, reaction
and diffusion in electrolysis, phase transitions, among others.

\medskip
We shall assume
\begin{eqnarray}
\label{REG1}
  A^{jk}_{\alpha\beta} \in C^2(\overline{\Omega} \times \R^N), \quad
  0< \lambda_0:= \inf
  \big\{A^{jk}_{\alpha\beta}(\mathbf{x},\mathbf{v)} \; \xi_\alpha^j \xi_\beta^k
  \big\},
\end{eqnarray}
where the infimum is taken over all $\xi \in S^{(Nd)-1}$,
($S^{(Nd)-1}$ denotes the unit sphere in $\R^{Nd}$), and
$(\mathbf{x},\mathbf{v}) \in \overline{\Omega}\times \R^N$. Also
\begin{eqnarray}
\label{REG2}
 \vp^j_\al \in C^2( \overline{\Omega} \times \R^N), \quad g_\al \in
 C^2( \overline{\Omega} \times \R^N).
\end{eqnarray}
The parabolic system \eqref{PS} is supplemented with an
initial-data
\begin{equation} \label{ID}
  \mathbf{u}(0,\mathbf{x})= \mathbf{u}_0(\mathbf{x}) \in C(\ol{\Omega}),
\end{equation}
and the following types of boundary-conditions on $\Gamma_T= (0,T)
\times \Gamma$: For some $0 \leq K \leq N$,  $K \in \N$ be fixed,
we set for $\mathbf{x}
\in \partial \Omega$
\begin{equation}
\label{BC}
\begin{aligned}
   \al= 1, \dots, K, \quad \big(f^j_\al \; n^j\big)(t,\mathbf{x})
   = {u_b}_\al(t,\mathbf{x}) \quad &\text{(Flux condition)},
   \\
   \al= K+1, \dots, N, \quad u_\al(t,\mathbf{x})
   = {u_b}_\al(t,\mathbf{x}) \quad &\text{(Dirichlet condition)},
\end{aligned}
\end{equation}
where ${u_b}_\al$ is a given function, and $\al= 1, \ldots, 0$ or
$\al= N+1, \ldots, N$, means clearly $\al\equiv 0$. The regularity
of $\mathbf{u}_b$ will be establish below, more precisely,
see Theorem \ref{LOCALEX} in Section \ref {GT} (general theory), 
Theorem  \ref{DIRICHLETTHM} 
in Section \ref{SECDIRICHLET}, 
for Dirichlet condition
of three-phase capillary-flow type systems.

\bigskip
The local existence of unique classical solution to the parabolic
system \eqref{PS}--\eqref{BC} might be proven either via fix point
arguments in H\"{o}lder space \cite{VSBTIZ}, or in weight
H\"{o}lder space \cite{TAAMPV1}, and also via semigroup theory in
$L^p$ space \cite{HA2}.
The important problem to answer is the question, whether this
local solution can be continued to be a global solution. It cannot
be expected that it is possible in all circumstances, as certain
counterexamples seen to indicate that solutions may start smoothly
and even remaind bounded, but develop a singularity after finite
time, see \cite{JSJO1,JSJO2}, and also \cite{TCPL}, \cite{ACJEZY}.
Although, in some papers, for instance in \cite{HA1}, \cite{AA},
\cite{VSB2}, \cite{HFVC} and \cite{MGGM} the global existence
result is proved under some structural conditions. This
information leads to the possibility to control some
lower-order norms  "a-priory".

\bigskip
In the first part of the paper, we introduce our strategy to study
global solutions to nondiagonal parabolic systems. We show in Section \ref{GT} that, 
classical solutions exist globally in time provided 
their orbits are pre-compact in the space of 
bounded and uniformly continuous functions.
First, we assume that $\mathbf{u}_0 \in E$, with $E \subset C(\ol{\Omega})$,
such that, in case of Dirichlet boundary condition, for each
$\mathbf{x} \in
\partial \Omega$, $u_{0\al}(\mathbf{x})= u_{b\al}(0,\mathbf{x})$,
$\al= K+1, \ldots,N$. In the papers \cite{TAAMPV1,TAAMPV2}, it was
proved the local existence and uniqueness of solution for
\eqref{PS}, \eqref{ID} and \eqref{BC}, with $\mathbf{u}_0 \in E$.
Also in the papers \cite{TAAMPV1,TAAMPV2}, it was proved the
continuous dependence of solutions of the initial data in $E$. Let
us write $\mathbf{u}(t,\mathbf{x};\mathbf{u}_0)$ the local
solution of the problem \eqref{PS}, \eqref{ID}, \eqref{BC} and
$[0,T_{\rm{max}})$ the maximal interval of the existence of
classical solution. If the set
$$
  D:= \{ \mathbf{u}(t,\mathbf{x});\, t \in (\ve,T_{\rm{max}})\},
$$
where $\ve < T_{\rm{max}}$ is any fixed positive number, is
pre-compact in $C(\ol{\Omega})$, that is to say, $\ol{D}$ is
compact in $C(\ol{\Omega})$, then the solution
$\mathbf{u}(t,\mathbf{x};\mathbf{u}_0)$ is global, i.e.
$T_{\rm{max}}= \infty$, which is proved in Theorem \ref{THMGES}.
We highlight that, this result is established in a general
context. Furthermore, if we have a priori estimate for
$\mathbf{u}(t,\mathbf{x};\mathbf{u}_0)$ in
$C^\gamma(\ol{\Omega})$, with $t \in (\ve,T_{\rm{max}})$, $\gamma
\in (0,1)$, then the local solution
$\mathbf{u}(t,\mathbf{x};\mathbf{u}_0)$ is global, see Corollary
\ref{CORGS} below.

\bigskip
Now, we recall an example from O. John and J. Star\'{a} \cite{JSJO2}.
For $d,N= 3$, $\kappa \in (0,4)$, the real analytic function
\begin{equation}
\label{SJS}
  \mathbf{u}(t,\mathbf{x};\mathbf{u}_0)= \frac{\mathbf{x}}{\sqrt{\kappa (1-t) + |\mathbf{x}|^2}}
\end{equation}
is, for each $t \in
[0,1)$, a classical solution of the system \eqref{PS}, with real analytic function $A^{jk}_{\alpha\beta}(\mathbf{u})$
(in a neighbourhood of $\ol{B(0,1)}$),
$\varphi_\alpha= g_\alpha= 0$, and respectively the initial and
boundary data
$$
  {\mathbf{u}_0}(\mathbf{x})= \frac{\mathbf{x}}{\sqrt{\kappa + |\mathbf{x}|^2}},
  \quad {\mathbf{u}_b}(t,\mathbf{x})= \frac{\mathbf{x}}{\sqrt{\kappa (1-t)+
  1}} \quad \text{if $|\mathbf{x}|= 1$}.
$$
This function \eqref{SJS} is bounded in $\ol{Q}_1$, but at time
$t=1$, i.e. $\mathbf{u}(1,\mathbf{x};\mathbf{u}_0)$ is not
continuous at $\mathbf{x}= 0$. Therefore, this example indicates
some sharpness for the result establish by Theorem \ref{THMGES}.
It is interesting to observe that, we have in this example from  O. John and
J. Stará \cite{JSJO2}, $\lambda_0 \simeq 0.04 \;
\lambda_1$, with $\lambda_1$ defined as
$$
  \lambda_1:= \sup \big\{ A^{jk}_{\alpha\beta}(\mathbf{x},\mathbf{v)} \; \xi_\alpha^j
  \xi_\beta^k \big\},
$$
where the supremum is taken over all $\xi \in S^{(Nd)-1}$, and
$(\mathbf{x},\mathbf{v}) \in \overline{\Omega}\times \R^N$.

\medskip
On the other hand, if $\lambda_0 \geq 0.33 \; \lambda_1$, $(d= 3)$, usually
called Cordes type conditions, then E. Kalita \cite{EK} proved that the
solution $\mathbf{u}(t)$ is bounded in $C^\gamma(\ol{\Omega})$.
Therefore, applying Corollary \ref{CORGS} we
obtain that the solution is global. This could be seen as a first
application of our strategy to answer the question whenever the
local solution can be continued to be a global solution. For
instance, this strategy solves the usual examples that come from
physics, where the matrix $A$ is a perturbation of the
identity, that is, $A= \lambda \, I_d + \mu B$, where
$\lambda>0$ is arbitrary and 
$\mu$ is a sufficiently small parameter.

\bigskip
In the second part of this paper, i.e. Section \ref{Three-phase capillary},
we apply our strategy to show
classical global solutions to nondiagonal parabolic systems, when
the matrix $A^{jk}_{\al\bt}$ is triangular (w.r.t. Greek indexes),
this is the major structural assumption. The motivation to study
such problems comes from three-phase capillary flow in porous
media. Therefore, we consider the parabolic system \eqref{PS}, with $N= 2$ and $d \geq 1$, that is 
\begin{equation} \label{PARABOLICCF}
   \begin{aligned}
   \partial_t u_1 +
   \divx \varphi_1(\mathbf{x},\mathbf{u})&= \divx \big(
   A_{1\bt}(\mathbf{x},\mathbf{u}) \nabla u_\bt \big)
   + g_1(\mathbf{x},\mathbf{u}),
   \\[5pt]
   \partial_t u_2 +
   \divx \varphi_2(\mathbf{x},\mathbf{u})&= \divx \big(
   A_{22}(\mathbf{x},\mathbf{u}) \nabla u_2 \big)
   + g_2(\mathbf{x},\mathbf{u}),
   \end{aligned}
\end{equation}
where $A^{jk}_{21}(\mathbf{x},\mathbf{u})= 0$ (major structural assumption).
Others conditions have to be considered, for instance see \eqref{FC}-\eqref{FCB}, which are used 
to establish the positively invariant regions (maximum principle), but we stress the 
following:  
\begin{equation}\label{COND1}
   \partial A^{jk}_{22} / \partial u_1 \equiv 0.
\end{equation}

Applying a different technic focused in a priori  estimates in H\"{o}lder spaces
and the Leray-Schauder's fixed-point theorem, H. Frid and V. Shelukhin \cite{HFVC} studied 
the homogeneous case $( f_\al=  f_\al(\mathbf{u}) )$ of parabolic system \eqref{PARABOLICCF}
in one dimension $(d= 1)$, with $g_\al \equiv 0$. In that paper, under the main condition $A_{21}(\mathbf{u}) \equiv 0$,
they proved existence and uniqueness of classical solution, see Theorem 1.1 (flux type condition), and 
existence of a classical solution with boundary condition assumed in the $L^2$-sense, see
Theorem 1.2 (Dirichlet condition).  Albeit, they have not considered the condition \eqref{COND1}, it seems to us that
this condition has not been avoided.

\bigskip
 Later S. Berres, R. Bürger and H. Frid
\cite{SBRBHF} to a similar (now $N \times N$) parabolic system cited before (that is, in \cite{HFVC}), they showed existence and uniqueness of classical solution, see 
Theorem 1.1 (perturbed flux condition), and
existence of a classical solution with boundary condition assumed in the $L^2$-sense, see
Theorem 1.2 (flux condition). Moreover, H. Amann \cite{HA1}
showded that, it
is sufficient to have an $L^\infty$ a priori bound with respect to
$\mathbf{x}$-variable and uniform Hölder continuity in time, with
$\gamma > d/(d+1)$,  to guarantee global existence.
We should mention that, the second part of
our work is neither contained in \cite{HA1} nor
\cite{SBRBHF,HFVC}.

\section{General theory} \label{GT}

\subsection{Local well-posedness} \label{SECLET}

Let us assume that the initial data of problem \eqref{PS},
\eqref{ID}, \eqref{BC} belongs to the space $E$, which is
constituted by vector functions
$$\mathbf{u}_0(\mathbf{x})=(u_{0_1}(\mathbf{x}),\ldots,u_{0_N}(\mathbf{x})),$$
such that ${u_0}_\al(\mathbf{x})$ belongs to
$C(\overline{\Omega})$. In case of Dirichlet boundary condition,
we assume also the agreement condition that $u_{0\al}(\mathbf{x})=
u_{b\al}(0,\mathbf{x})$, for $\mathbf{x} \in
\partial \Omega$, and $\al= K+1, \ldots,N$ holds for
$u_\al(t,\mathbf{x})$. Therefore, we impose no additional
assumption in the case of flux (boundary) condition \eqref{BC} for
$u_\al(t,\mathbf{x})$, except the containment of
$u_{0\al}(\mathbf{x})$ in space $C(\overline{\Omega})$. The local
existence theorem of problem \eqref{PS}, \eqref{ID}, \eqref{BC}
with initial data in $E$ is proved in \cite{TAAMPV1}, see also
\cite{TAAMPV2}.

\bigskip
In order to prove a local existence theorem with initial data in
$E$, we need to use the estimates of linear parabolic systems in
weighted H\"{o}lder classes, obtained by V. Belonosov and T.
Zelenjak in one dimensional case \cite{VSBTIZ}, and simultaneously
by V. Belonosov \cite{VSB1,VSB2}, also V. Solonnikov and A.
Khachatryan \cite{VASAGK} for parabolic system in several space
variables. Let us present these classes as they are related to the
case under discussion. Let $f(t,\mathbf{x})$ be a real function
defined in $Q_T$. Denote
$$
  \triangle_y^x f= f(t,\mathbf{x}) - f(t,\mathbf{y}),
  \quad
  \triangle^t_\tau f= f(t,\mathbf{x}) - f(\tau,\mathbf{x}),
$$
and suppose $s\geq 0,\quad r\leq s$.  Given a function
$u(t,\mathbf{x})$ which is defined and continuous in $Q'_T:= (0,T]
\times \overline{\Omega}$ together with its derivatives $D^\mu_t
D^\nu_{\mathbf{x}} u $ of order $2 \mu + |\nu| \leq s$, $\nu=
(\nu_1,\ldots,\nu_N)$, where $\nu_\alpha$, $(\alpha = 1,\ldots,N)$
are nonnegative integers, with $\mid\nu\mid= \Sigma_{\alpha=1}^N
\nu_\alpha$. Then, we introduce the following semi-norms
$$
  [u]_{r,s;\mathbf{x}}^{Q_T}= \sup \Big\{
  t^{\frac{s-r}{2}}
  \frac{\big| \triangle^x_y D^\mu_t D^\nu_{\mathbf{x}} u \big|}
  {|x-y|^{s-[s\,]}} \Big\},
$$
and
$$
  [u]_{r,s;t}^{Q_T}= \sup \Big\{ \theta^{\frac{s-r}{2}} \frac{\big| \triangle^t_\tau
  D^\mu_t D^\nu_{\mathbf{x}} u  \big|}{|t-\tau|^{\frac{(s-2\mu-\mid \nu
  \mid)}{2}}} \Big\},
$$
where $\theta= \min\{\tau,t\}$. The supremum in the first
semi-norm is taken over all points $(t,\mathbf{x}) \neq
(t,\mathbf{y})$ of $Q'_T$, and all $\mu$, $\nu$ such that $2 \mu +
|\nu|= [s]$. The second supremum is taken over all points
$(t,\mathbf{x}) \neq (\tau,\mathbf{x})$ of $Q'_T$, and $\mu$,
$\nu$ such that $0< s-2 \mu- |\nu|< 2$. Here, $[s\,]$ denotes the
largest integer which is smaller or equal to $s$.

\medskip
In addition, for $k\geq 0$ be an integer and each $r \leq k$, we
set
$$
  |u|^{Q_T}_{r,k}= \sup \big\{t^{\frac{(k-r)}{2}} \; |D^\mu_t
  D^\nu_{\mathbf{x}} u| \big\},
$$
where the supremum is taken over all $(t,\mathbf{x})\in Q'_T$ and
all values of $\mu$, $\nu$ such that, $0<s-2\mu - |\nu|= k$.
Finally, we define
$$
  [u]^{Q_T}_{r,s}:=[ u ]^{Q_T}_{r,s;\mathbf{x}}+[u]^{Q_T}_{r,s;t}.
$$
We observe that, the symbols $[u]^{Q_T}_s$ and $|u|^{Q_T}_k$ will
denote respectively the seminorms resulting from $[u]^{Q_T}_{r,s}$
and $|u|^{Q_T}_{r,k}$, when $r=s$ and $r=k$.

\bigskip
Let $H^s(Q_T)$ be the space of functions $u(t,\mathbf{x})$ having
continuous partial derivatives $D^\mu_tD^\nu_{\mathbf{x}} u$ of
order $2\mu+ |\nu| \leq s$ in $Q_T$, and the finite norm
$$
  \|u \|^{Q_T}_s= [u]^{Q_T}_s + \sum_{0\leq k<s} |u|^{Q_T}_k.
$$
This space is also usually denoted by $ H^{s,s/2}_{\mathbf{x},t}$
or $C^{s,s/2}_{\mathbf{x},t}$, but it is more convenient for us to
use the abridged notation $H^s$. We note that, if $u$ does not
depend on $t$, then the quantity $\|u\|^{Q_T}_s$ is converted into
the norm $\|u\|^{\Omega}_s$ in the standard H\"{o}lder space
$C^s(\Omega)$.

\medskip
Now suppose $s \geq 0$ and $r \leq s$. When $r$ is not integer,
the space $H^s_r(Q_T)$ is the set of functions $u(t,\mathbf{x})$
having continuous derivatives $D^{\mu}_t D^{\nu}_{\mathbf{x}} u$
of order $2\mu+ |\nu| \leq s$ in $Q'_T$, and the finite norms
$$
\begin{aligned}
  \|u \|^{Q_T}_{r,s}&=[u]^{Q_T}_{r,s} +
  \sum_{r<k<s} |u|^{Q_T}_{r,s} + \| u \|^{Q_T}_r,
  \\[5pt]
  \|u \|^{Q_T}_{r,s}&=[u]^{Q_T}_{r,s} + \sum_{0 \leq
  k<s} |u|^{Q_T}_{r,s},
\end{aligned}
$$
respectively when $r \geq 0$, and $r<0$. Moreover, when $r$ is an
integer the space $H^s_r(Q_T)$ is defined as the completion of
$H^s(Q_T)$ with respect to the above norm.

\medskip
Finally, the space $H^s_r(\Gamma_T)$ of functions defined on the
lateral surface of the cylinder $Q_T$ is defined as the set of
traces on $\Gamma_T$ of functions in $H^s_r(Q_T)$. The norm in
this space is defined by the following quantity
$$
  \| \varphi \|^{\Gamma_T}_{r,s}=
  \inf \| \Phi \|^{Q_T}_{r,s},
$$
where the infimum is taken over all $\Phi \in H^s_r(Q_T)$
coinciding with $\varphi$ on $\Gamma_T$.

\bigskip
Our first result is based on two theorems which are proven in
\cite{TAAMPV1} (see also \cite{TAAMPV2}): 

\medskip
\begin{theorem}$($Local existence with initial data from E$)$
\label{LOCALEX} Suppose that the problem \eqref{PS}, \eqref{ID}
and \eqref{BC} satisfies assumptions \eqref{REG1}, \eqref{REG2},
with $\mathbf{u}_0\in E$ and ${u_b}_\al \in
H_0^{\gamma(\alpha)}(\Gamma_T)$,
$$
  \gamma(\alpha)= \left \{
  \begin{aligned}
  1+\gamma, &\quad \alpha= 1, \ldots, K,
  \\
  2+\gamma, &\quad \alpha= K+1, \ldots, N,
  \end{aligned}
\right.
$$
with $0<\gamma <1$. Then, there exists a positive number $T$ such
that, the problem \eqref{PS}, \eqref{ID} and \eqref{BC} has a
unique classical solution
$$\mathbf{u}(t,\mathbf{x};\mathbf{u}_0) \in
H_0^{2+\gamma}(\ol{Q}_T).$$ The positive constant $T$ depends on
the value $\| \mathbf{u}_0 \|^{\Omega}_0$ and the modulus of
continuity of the function $\mathbf{u}_0(x)$.
\end{theorem}

\begin{theorem} $($The continuous dependence of solutions$)$
\label{CDSID}
Suppose that all conditions of Theorem \ref{LOCALEX} are
satisfied, and $\mathbf{u}(t,\mathbf{x};\mathbf{u}_0)$ is the classical 
solution to problem \eqref{PS}--\eqref{BC} in $Q_T$ given by
Theorem \ref{LOCALEX}. Then, there exists a positive number
$\delta>0$, such that, if $\mathbf{v}_0 \in E$, ${v_b}_\al \in
H_0^{\gamma(\al)}$, satisfy
$$
  \| \mathbf{u}_0-\mathbf{v}_0
\|^{\Omega}_0 + \|{u_b}_\al -
{v_b}_\al\|_{0,\gamma(\al)}^{\Gamma_T} < \delta
$$
and also the compatibility conditions, then there exists a
classical solution $\mathbf{v}(t,\mathbf{x};\mathbf{v}_0)$ to
problem \eqref{PS}--\eqref{BC} in $Q_T$, with the initial and
boundary data respectively $\mathbf{v}_0$ and $\mathbf{v}_b$.
Moreover, it follows that
\begin{eqnarray}
\|
\mathbf{u}(t,\mathbf{x};\mathbf{u}_0)-\mathbf{v}(t,\mathbf{x};\mathbf{v}_0)
\|^{Q_T}_{0,2+\gamma}\leq C \; \| \mathbf{u}_0-\mathbf{v}_0
\|^{\Omega}_0 + \|{u_b}_\al -
{v_b}_\al\|_{0,\gamma(\al)}^{\Gamma_T},
\end{eqnarray}
where the constant $C$ does not depend on $\mathbf{v}_0$ and
$\mathbf{v}_b$.
\end{theorem}

We remark that, the proofs of the above theorems are slight 
modifications to that ones, respectively given for Theorem 1
and Theorem 2  in \cite{TAAMPV1}.

\medskip
\begin{remark}
1. First, the example recalled from \cite{JSJO2} shows that, we may
not neglect the dependence of $T>0$ on the modulus of continuity
of $\mathbf{u}_0$.

\medskip
2. We can consider the problem \eqref{PS},  \eqref{BC} and the
initial condition $\mathbf{u}_0(\mathbf{x})=
\mathbf{u}(\tau,\mathbf{x})$, with ${u_0}_\al(\mathbf{x})=
{u_b}_\al(\tau,\mathbf{x})$ for $\al= K+1,\ldots,N$, and each $\mathbf{x} \in \partial \Omega$, such that, the existence time
interval does not depend on  $\tau \in [0,T]$.
\end{remark}

\subsection{Global solutions}\label{GlobalSolutions}

To prove the global solvability of the problem \eqref{PS},
\eqref{ID} and \eqref{BC} in the space $H_0^{2+\gamma}(\ol{Q}_T)$,
we use the result from Section \ref{SECLET} related to time local
classical solvability of problem problem \eqref{PS}, \eqref{ID}
and \eqref{BC} in $H_0^{2+\gamma}(\ol{Q}_T)$. Denote by
$[0,T_{\rm{max}})$ the maximal existence interval for the smooth
solution $\mathbf{u}(t,\mathbf{x})$. Then, for all $t \in (\ve,T_{\rm{max}})$, we have
$\mathbf{u}(t) \in C^{2+\gamma}(\ol{\Omega})$, and $|\mathbf{u}|
\leq M$.

\begin{theorem}\label{THMGES}
If $\ol{D}$ is compact in $C(\ol{\Omega})$, then $T_{\rm{max}}=
\infty$ or equivalently, for any $T>0$ the problem \eqref{PS},
\eqref{ID} and \eqref{BC} has global solution in $Q_T$.
\end{theorem}

\begin{proof}
First, let us suppose that $T_{\rm{max}} < \infty$. Then, we set
$\mathbf{u}_{n}(\mathbf{x})= \mathbf{u}(t_{n},\mathbf{x})$, with
$t_{n} \to T_{\rm{max}}$ (monotonically crescent). Since $\ol{D}$ is
compact by hypothesis, hence the sequence contains a subsequence
$\mathbf{u}_{n_i}(\mathbf{x})$, which converges in
$C(\ol{\Omega})$. Denote this subsequence also by
$\mathbf{u}_{n}(\mathbf{x})$, and
$$
  \mathbf{u}_{n}(\mathbf{x}) \to
  \tilde{\mathbf{u}}_{0}(\mathbf{x})
  \quad \text{in $C(\ol{\Omega})$}.
$$
From Theorem \ref{LOCALEX} there exists $\sigma>0$, such that the
problem \eqref{PS}--\eqref{BC} has a classical solution
$\mathbf{u}(t,\mathbf{x};\tilde{\mathbf{u}}_0)$ if $t \in
[T_{\rm{max}},T_{\rm{max}}+\sigma]$. Moreover, it follows from Theorem
\ref{CDSID} that, for $n$ sufficiently large the classical
solution of \eqref{PS}, \eqref{ID} and \eqref{BC} (i.e. with
initial data $\mathbf{u}_{n}(\mathbf{x})=
\mathbf{u}(t_{n},\mathbf{x})$) exists if $t \in [t_n,
t_n+\sigma]$. Therefore, the function
$$
  \mathbf{u}(t,\mathbf{x})=
  \left \{
  \begin{aligned}
    \mathbf{u}(t,\mathbf{x;\mathbf{u}_0}) &\quad \text{if $t \in
    [0,t_{n}]$},
    \\
    \mathbf{u}(t,\mathbf{x};\mathbf{u}_{n}) &\quad \text{if $t \in
    [t_{n},t_n+\sigma]$},
  \end{aligned}
  \right.
$$
is a classical solution of \eqref{PS}--\eqref{BC}. Consequently,
for $n$ sufficiently large, $t_{n}+\sigma > T_{\rm{max}}$, which
is a contradiction.
\end{proof}

\begin{corollary}
\label{CORGS} If $\gamma> 0$ and $\ol{D}$ is bounded set in
$C^\gamma(\ol{\Omega})$, then the problem \eqref{PS}, \eqref{ID}
and \eqref{BC} has global solution in $Q_T$.
\end{corollary}

\smallskip
For each $\varepsilon >0$, we denote hereafter $Q_T^{\varepsilon}:= (\varepsilon,T) \times \Omega$, 
also $\Gamma_T^{\varepsilon}:= (\varepsilon,T) \times \Gamma$. Then,
another consequence of Theorem \ref{THMGES} is the following

\begin{proposition} \label{PROPREG}
Let $T>0$ be arbitrary and conditions \eqref{REG1}, \eqref{REG2} hold. Assume
that $u_\alpha(t,\mathbf{x})$ is a classical solution for 
the initial-boundary value problem
\eqref{PS}, \eqref{ID} and \eqref{BC} in $Q_T$, with $u_\alpha \in C(\overline{Q}_T)$. Then, 
there exists $\gamma \in (0,1)$, such that,
the unique classical solution exists globally in time in
$\overline{Q}_T$, and for each $\varepsilon>0$, $u_\alpha \in C^{2+\gamma,
1+\frac{\gamma}{2}}_{x, \quad t}(\overline{Q^\varepsilon_T})$.

\end{proposition}

\begin{proof}
The proof follows from Theorem \ref{THMGES} and the definition of
the weighted spaces.
\end{proof}

\section{Three-phase capillary-flow type systems}\label{Three-phase capillary}

In this part of the article, we establish the existence (global)
and uniqueness theorem for the IBVP \eqref{PS}--\eqref{BC}, when
$N= 2$ and the system admits some additional conditions, to be
precisely below. 
Here to avoid more technicality, we assume zero-flux boundary condition.
This system is motivated by one-dimensional
three-phase capillary flow through porous medium (e.g. oil, water
and gas), which is related for instance to planning operation of
oil wells. Here, we are not going to enter in more physical
details. The interesting reader is addressed to H. Frid and V.
Shelukhin \cite{HFVC}, in order to obtain more information about
the laws of multi-phase flows in a porous medium, and also
applications.

\bigskip
First, we define for $\eta \in \R^d$
\begin{equation}
\label{DLS}
    \begin{aligned}
  \Lambda_{\al\bt}(\mathbf{x},\mathbf{u},\eta) &:=
  A_{\al\bt}^{jk}(\mathbf{x},\mathbf{u}) \; \eta^j \; \eta^k,
\\
 \Lambda_{\al\bt}^\mathbf{n}(\mathbf{x},\mathbf{u},\eta) &:=
  A_{\al\bt}^{jk}(\mathbf{x},\mathbf{u}) \; n^j \; \eta^k.
  \end{aligned}
\end{equation}
Then, we have the following
\begin{definition}
The family $\{A_{\al\bt}^{jk}\}$ is called {\bf normally elliptic}
on $Q_T$, when $\\$ i) For each $(t,\mathbf{x}) \in \ol{Q}_T$, and
all $\eta \in S^{d-1}$
\begin{equation}
\label{ELPCD1}
  \sigma\Big(\Lambda_{\al\bt}(\mathbf{x},\mathbf{u},\eta) \Big)
  \subset \big[Re (z) > 0\big] \equiv \{ z \in
  \C: Re (z)> 0\},
\end{equation}
where $\mathbf{u}= \mathbf{u}(t,\mathbf{x})$,  and $\sigma(M)$ denotes the
spectrum of the matrix $M$. $\\$
\\
ii) For each $r \in \Gamma$, $\xi \in T_r \Gamma$, $|\xi|= 1$ and
$\omega \in \big[Re (z) \geq 0\big]$, with $\omega \neq 0$, zero
is the unique exponentially decaying solution of the BVP
\begin{equation}
\label{EP}
  \begin{aligned}
  \big[ \omega + \Lambda_{\al\bt}(\mathbf{x},\mathbf{u}, \eta + \mathbf{n} \; i \partial_t) \big] \mathbf{u}&=
  0 \qquad \text{on $\Gamma_T$},
  \\[5pt]
   \Lambda^\mathbf{n}_{\al\bt}(\mathbf{x},\mathbf{u}, \eta + \mathbf{n} \; i \partial_t) \mathbf{u}_0&=
   0 \qquad \text{on $\Gamma$},
  \end{aligned}
\end{equation}
where $i$ is the imaginary number. Moreover, the system \eqref{PS}
is said parabolic in Petrovskii-sense and admits the Lopatinski's
compatibility condition, when the family $\{A_{\al\bt}^{jk}\}$
respectively satisfies \eqref{ELPCD1} and \eqref{EP} (see, for instance,  \S 8
Chapter VII in Ladyzenskaja, Solonikov and Ural'ceva
\cite{OALVASNNU}).
\end{definition}

In particular, a simple case of normally elliptic family occurs
when the family $\{A_{\al\bt}^{jk}\}$ is triangular w.r.t. the
Greek indexes, for instance upper triangular. Hence in this case,
$\Lambda$ and $\Lambda^\mathbf{n}$ are also upper triangular and
conditions \eqref{ELPCD1}, \eqref{EP} holds true if, and only if,
there exists $\mu_0> 0$, such that
\begin{equation*}
\label{PLC}
   \Lambda_{\al\al}(\mathbf{x},\mathbf{u},\eta) \geq \mu_0,
\end{equation*}
which is satisfied in our case, since from assumption \eqref{REG1}
we have
$$
  \Lambda_{\al\al}(\mathbf{x},\mathbf{u},\eta)=
  A^{jk}_{\al\al}(\mathbf{x},\mathbf{u}) \, \eta^j \, \eta^k
   \geq \la_0 > 0.
$$
Therefore, from this point we shall consider
\begin{equation}
\label{AIJ21}
  A_{21}^{jk}(\mathbf{x},\mathbf{u}) \equiv 0.
\end{equation}

\medskip
Now, we consider the following domain
\begin{equation}
\label{UB1}
   B_\Delta:= \big\{(u_1,u_2) \in \R^2 : 0 \leq
   u_1,u_2 \leq 1, \; u_1 + u_2 \leq 1 \big\},
\end{equation}
which is motivated by the applications. Then, we assume 
\begin{equation}
\label{IBCBD}
    \mathbf{u}_0(\mathbf{x}), \mathbf{u}_b(t,\mathbf{x})
    \in B_\Delta \quad \text{for each
$\mathbf{x} \in \ol{\Omega}$ and $t \geq 0$}.
\end{equation}
Following Serre's book \cite{DS1}, Vol 2 (Chapter 1), we may have
the triangle $B_\Delta$ written also by the intersection of the
following functions
$$
  G_1(\mathbf{u})= -u_1, \quad G_2(\mathbf{u})= -u_2,
  \quad G_3(\mathbf{u})= -1 + u_1 + u_2,
$$
that is,
$$
  B_\Delta \equiv \bigcap_{h=1}^3 \big\{\mathbf{u} \in \R^2 : G_h(\mathbf{u}) \leq 0
  \big\},
  \quad
  \partial B_\Delta= \big\{\mathbf{u} \in \ol{B}_\Delta : G_h(\mathbf{u})= 0, h=1,2,3\big\}.
$$
We seek for functions $\mathbf{u}(t,\mathbf{x})$ solutions of the
IBVP \eqref{PS}--\eqref{BC}, such that $\mathbf{u}(t,\mathbf{x})
\in B_\Delta$, which is satisfied under some additional hypotheses
on $\varphi_\alpha^j$, $A_{\alpha \beta}^{jk}$ and $g_\alpha$,
which is to say, for each $\mathbf{x} \in \ol{\Omega}$, and all
$\mathbf{u} \in
\partial B_\Delta$, ($h= 1,2,3$, here no summation on indices
$h$)
\begin{equation}
\label{FC}
  \frac{\partial \varphi_\al^j(\mathbf{x},\mathbf{u})}{\partial u_\bt}
  \; \frac{\partial G_h(\mathbf{u})}{\partial u_\al}=
  \la^j_h(\mathbf{x},\mathbf{u}) \; \frac{\partial G_h(\mathbf{u})}{\partial
  u_\bt},
\end{equation}
\begin{equation}
\label{LC}
    A^{jk}_{\al\bt}(\mathbf{x},\mathbf{u})
  \; \frac{\partial G_h(\mathbf{u})}{\partial u_\al}=
  \mu^{jk}_h(\mathbf{x},\mathbf{u}) \; \frac{\partial G_h(\mathbf{u})}{\partial
  u_\bt},
\end{equation}
\begin{equation}
\label{GGC}
 \big( g_\al(\mathbf{x},\mathbf{u}) -
 \gamma_\al(\mathbf{x},\mathbf{u}) \big)
 \; \frac{\partial G_h(\mathbf{u})}{\partial u_\al}
  \leq 0,
\end{equation}
%
%
\begin{equation}
\label{FCB}
 \big(\vp_\al^j \; n^j \big)|_{\{u_\al= 0\}}=
 \big((\vp_1^j+\vp_2^j) \; n^j \big)|_{\{u_1+u_2= 0\}} \equiv 0,
\end{equation}
for some functions $\la^j_h$ and $\mu^{jk}_h$, where $\gamma_\alpha(\mathbf{x},\mathbf{v}):= 
\partial_{x_j} \varphi_\alpha^j(\mathbf{x},\mathbf{v})$.
Remark that, for
$h= 1,2,3$, $\mu^{jk}_h \eta^j \eta^k > 0$. 

\medskip
Due to
assumptions \eqref{FC}--\eqref{FCB} we can adapt the theory of
(positively) invariant regions for nonlinear parabolic systems
developed by Chuey, Conley and Smoller \cite{CCS}, see also the
cited book of Serre and the Smoller's book \cite{JS}. One remarks
that, conditions \eqref{FC}, \eqref{LC} holds true for instance in
three-phase capillary flow in porous media (black oil model),
further in this case, \eqref{GGC} is trivially satisfied.

\medskip
Under the above considerations, let us assume that the solution $u_\al(t,\mathbf{x})$
of the system \eqref{PS}--\eqref{BC} exists for each $t \in [0,T_{\rm{max}})$,
with $T_{\rm{max}}< \infty$. If we show that $u_\al(t,\mathbf{x}) \in C^\gamma(\bar{Q}^\varepsilon_{T_{\rm{max}}})$,
$\gamma \in (0,1)$, then from Corollary \ref{CORGS}, it follows that $T_{\rm{max}}= \infty$, which is a contradiction.
Hence $T_{\rm{max}}= \infty$ and we have classical global existence theorem. In fact, to establish a global
result we need first a uniform estimate, which will be given in the next section.

\subsection{Positively invariant regions} \label{IBVP}

We consider the parabolic system \eqref{PARABOLICCF} with the initial-boundary data \eqref{ID} and \eqref{BC}.
Let $\mathbf{v}= (v_1,v_2)$ be a constant vector, such that
$v_1,v_2 >0$ and $v_1+v_2 < 1$. For each $\epsilon> 0$, we
regard the following approximated system 
\begin{equation}
\label{PSP}
\begin{aligned}
   \partial_t u^\epsilon_\al +
   \divx \varphi_\al(\mathbf{x},\mathbf{u}^\epsilon)&= \divx \big(
   A_{\al\bt}(\mathbf{x},\mathbf{u}^\epsilon) \nabla u^\epsilon_\bt \big) + g^\epsilon_\al(\mathbf{x},\mathbf{u}^\epsilon),
   \end{aligned}
\end{equation}
with the initial condition in $\Omega$
\begin{equation}
\label{ICA}
\begin{aligned}
    \mathbf{u}^\epsilon(0,\mathbf{x})= \mathbf{u}^\epsilon_0(\mathbf{x}),
\end{aligned}
\end{equation}
and the boundary conditions on $\Gamma_T$
\begin{equation}
\label{BCA}
\begin{aligned}
   \big(\; \varphi_\al^j(\mathbf{x},\mathbf{u}^\epsilon)
   - A^{jk}_{\al\bt}(\mathbf{x},\mathbf{u}^\epsilon) \; \frac{\partial u^\epsilon_\bt}{\partial x_k} \;
    \big)\; n^j &= \epsilon \;  A^{jk}_{\al\bt}(\mathbf{x},\mathbf{u}^\epsilon)  \; n^j \; \eta^k
   (u^\epsilon_\bt-u^\epsilon_{b\bt}),  
    \\[5pt]
   \text{or} \quad u^\epsilon_\al &= u^\epsilon_{b\al},
\end{aligned}
\end{equation}
where 
$g^\epsilon_\al(\mathbf{x},\mathbf{u}^\epsilon)= g_\al(\mathbf{x},\mathbf{u}^\epsilon)    
   + \epsilon \big(v_\al - u^\epsilon_\al \big)$, 
$\mathbf{u}^\epsilon_0:= (1-\epsilon) \; \big( \mathbf{u}_0 + \epsilon/2
\big)$ and similarly $\mathbf{u}^\epsilon_b:= (1-\epsilon) \; \big(
\mathbf{u}_b + \epsilon/2 \big)$.



\bigskip
Let $T> 0$ be given. In order to show the existence of positively invariant regions, we assume that, for each $\epsilon>0$, there exists
$u_\alpha^\epsilon \in C^{2,1}(Q_T)$ the unique solution to \eqref{PSP}--\eqref{BCA}, see Proposition \ref{PROPREG}. Then, we have the following 

\begin{lemma}\label{PIR}
Let $T> 0$ be given. 
If $\mathbf{u}_0$, $\mathbf{u}_b$ take values in $B_\Delta$,
then the unique solution $\mathbf{u}^\epsilon$ of \eqref{PSP}--\eqref{BCA} take values in
the interior of $B_\Delta$, i.e. $\mathbf{u}^\epsilon(t,\mathbf{x}) \in {\rm int}
B_\Delta$, for each $(t,\mathbf{x}) \in \ol{Q_T}$.
\end{lemma}

\begin{proof}
For simplicity, we drop throughout the proof the superscript
$\epsilon$, whenever it is not strictly necessarily. First, let us denote $Z_h(t,\mathbf{x})= G_h(\mathbf{u}(t,\mathbf{x}))$,
$(h=1,2,3)$. Since $\mathbf{u}^\epsilon_0 \in {\rm int} B_\Delta$, we have
$$
  \sup_{\mathbf{x} \in \Omega} Z_h(0,\mathbf{x}) < 0, \qquad (h= 1,2,3).
$$
By contradiction, suppose that there exists $t_0 > 0$, the first
time, such that
$$
  Z_h(t_0,\mathbf{x}_0):= \sup_{\mathbf{x} \in \Omega} Z_h(t_0,\mathbf{x}) = 0,
$$
for some $h= 1, 2$ or $3$. That is to say, $\mathbf{u}(t_0,\mathbf{x}_0)
\in \partial B_\Delta$, where $\mathbf{x}_0 \in \ol{\Omega}$ and
$$
  \sup_{\mathbf{x} \in \Omega} Z_h(t,\mathbf{x}) < 0 \quad \text{for all $0 \leq t < t_0$}.
$$

If $\mathbf{x}_0 \in \Omega$, then it follows by \eqref{PSP} at $(t_0,\mathbf{x}_0)$
that
$$
  \partial_t u_\al + \frac{\partial \vp_\al^j}{\partial u_\bt} \;
  \frac{\partial u_\bt}{\partial x_j} + \gamma_\al =
  \frac{\partial}{\partial x_j} \Big(A_{\al\bt}^{jk}(\mathbf{x},\mathbf{u})
   \frac{\partial u_\bt}{\partial x_k}\Big)
  + g_\al + \epsilon \big(v_\al - u_\al \big).
$$
Now, we multiply the above equation by $\partial G_h/\partial
u_\al$ and applying conditions \eqref{FC} and \eqref{LC}, we have
\begin{equation}
\label{LPIR}
  \begin{aligned}
  \partial_t Z_h + \la^j_h \;
  \frac{\partial Z_h}{\partial x_j} &=
  \frac{\partial \mu^{jk}_h}{\partial x_j} \; \frac{\partial Z_h}{\partial
  x_k} + \mu^{jk}_h \; \frac{\partial^2 Z_h}{\partial x_k \partial x_j}
\\[5pt]
  &\quad+ \frac{\partial G_h}{\partial u_\al} (g_\al-\gamma_\al)
  + \epsilon \; \frac{\partial G_h}{\partial u_\al}\big(v_\al - u_\al \big).
  \end{aligned}
\end{equation}
By assumption of the contradiction, we must have at $(t_0,\mathbf{x}_0)$
$$
  \partial_t Z_h \geq 0, \quad \nablax Z_h= 0,
  \quad \deltax Z_h \leq 0,
$$
hence by \eqref{LPIR}, using \eqref{GGC} and since $\mu_h^{jk}
\eta^j \eta^k > 0$, we obtain the following contradiction
$$
  0 \leq \partial_t Z_h \leq \epsilon \; \frac{\partial G_h}{\partial u_\al} \big(v_\al - u_\al \big)
  < 0.
$$
Indeed, suppose for instance $Z_1(t_0,\mathbf{x}_0)= 0$, hence
$u_1(t_0,\mathbf{x}_0)= 0$ and then, we have
$$
  0 \leq \partial_t Z_1 \leq \epsilon \; (-1) \big(v_1 - 0 \big)< 0.
$$
Analogous result when $h= 2$. For $h= 3$, we have
$u_1(t_0,\mathbf{x}_0)+u_2(t_0,\mathbf{x}_0)= 1$, therefore
$$
  0 \leq \partial_t Z_3 \leq \epsilon \; \big(v_1+v_2 - 1\big)< 0.
$$

\medskip
It remains to show the case $\mathbf{x}_0 \in \Gamma$, and
let us first consider the simple case of Dirichlet boundary condition. 
Indeed, the cases $h=1,2$ are trivial, which is to say, we obtain the contradiction  $u_{b\alpha} < 0$.
For $h= 3$, it follows that
$$
  \begin{aligned}
  0 &=
  (1-\epsilon) (u_{b1} + \epsilon/2)  - u_1 + (1-\epsilon) (u_{b1} + \epsilon/2)  -
  u_2
  \\
  &= (u_{b1} + u_{b2} - 1) + (\epsilon - \epsilon (u_{b1} + u_{b2})) -
  \epsilon^2.
  \end{aligned}
$$
It is enough to consider the limit possibilities, that is, $u_{b1}
+ u_{b2}= 0$ and $u_{b1} + u_{b2}= 1$, and the result follows.
Finally, we consider the flux condition. 
Recall that, 
the flux condition is given by 
$$
   \big(\; \varphi_\al^j(\mathbf{x},\mathbf{u})
   - A^{jk}_{\al\bt}(\mathbf{x},\mathbf{u}) \; \frac{\partial u_\bt}{\partial x_k} \;
    \big)\; n^j= \epsilon \; A^{jk}_{\al\bt}(\mathbf{x},\mathbf{u}) \; n^j \; \eta^k
   (u_\bt-u_{b\bt}).
$$
Thus multiplying the above equation by $\partial G_h/\partial
u_\al$, we obtain
$$
  \mu^{jk}_h \; n^j \; \frac{\partial Z_h}{\partial x_k}= \epsilon \;
  \mu_h^{jk} \; n^j \; \eta^k \;
  \frac{\partial G_h}{\partial u_\bt} \; (u_{b\bt} - u_\bt),
$$
where we have used \eqref{FCB}. 
From a similar argument used before for
Dirichlet condition, we derive the
following contradiction
$$
  0 \leq \mu^{jk}_h \; n^j \; \frac{\partial Z_h}{\partial x_k} <
  0,
$$
which finish the proof.
\end{proof}

\begin{corollary}
\label{CORIR}
Since the estimates in Lemma \ref{PIR} for $\mathbf{u}^\epsilon(t,\mathbf{x})$ depends continuously on $\epsilon>0$,
therefore passing to the limit as $\epsilon$ goes to $0^+$, we obtain
that $\mathbf{u}(t,\mathbf{x}) \in B_\Delta$.
\end{corollary}

\subsection{Main results}

In this section we consider the problem \eqref{PSP}-\eqref{BCA}
with $\epsilon= 0$, that is
\begin{equation}
\label{PSPZERO}
\begin{aligned}
   \partial_t u_\al +
   \divx \varphi_\al(\mathbf{x},\mathbf{u})&= \divx \big(
   A_{\al\bt}(\mathbf{x},\mathbf{u}) \nablax u_\bt \big)
   + g_\al(\mathbf{x},\mathbf{u}),
   \end{aligned}
\end{equation}
the initial condition in $\Omega$
\begin{equation}
\label{ICAZERO}
\begin{aligned}
    \mathbf{u}(0,\mathbf{x})= \mathbf{u}_0(\mathbf{x}),
\end{aligned}
\end{equation}
and the boundary conditions on $\Gamma_T$
\begin{equation}
\label{BCAZERO}
\begin{aligned}
   &\big(\; \varphi_\al^j(\mathbf{x},\mathbf{u})
   - A^{jk}_{\al\bt}(\mathbf{x},\mathbf{u}) \; \frac{\partial u_\bt}{\partial x_k} \;
    \big)\; n^j= 0, \quad
   \text{or} \quad u_{\al}= u_{b\al}.
\end{aligned}
\end{equation}

Under conditions \eqref{IBCBD}--\eqref{FCB}, we proved that
$\mathbf{u} \in B_\Delta$.
To prove the global solvability of problem \eqref{PSPZERO}--\eqref{BCAZERO},
we use first Theorem \ref{LOCALEX} about time-local classical solvability of this problem
in $H^{2+\gamma}_0(Q_{T})$.
Denote by $[0,T_{\rm{max}})$ the maximal existence interval of solution
from Theorem  \ref{LOCALEX}. Let us assume that $T_{\rm{max}} < \infty$.
If we prove that $u_\alpha \in C^\gamma(Q^\varepsilon_{T_{\rm{max}}})$ and
$\|\mathbf{u}(t)\|_\gamma^{\bar{\Omega}} \leq C$, with $t \in [\varepsilon, T_{\rm{max}})$,
it follows from the proof of Theorem \ref{THMGES} that there exists $\delta>0$, such that,
the solution $\mathbf{u}(t,\mathbf{x})$ exists on $[0,T_{\rm{max}}+\delta)$, which is a contradiction
of the definition of $T_{\rm{max}}$.

\medskip
Let us denote
by $V^{1,0}_2(Q_T)$ the space of functions with finite norm
$$
  \sup_{[0,T]} \|u(t,\cdot)\|_{2,\Omega} + \|\nablax u\|_{2,Q_T}.
$$
It shall be understood that, the vector-value function
$\mathbf{u}$ belongs to $V^{1,0}_2(Q_T)$, if $ u_\al \in
V^{1,0}_2(Q_T)$ for $\al= 1,2$.

\begin{lemma}
\label{MainLemma}
Consider the initial-boundary value problem  \eqref{PSPZERO}-\eqref{BCAZERO}, 
and assume the conditions \eqref{AIJ21}, \eqref{IBCBD}--\eqref{FCB}. If $\mathbf{u}_0 \in E$
and $\mathbf{u}_b \in H_0^{2+\gamma}(\Gamma_T)$,  with $\gamma \in (0,1)$, then
$u_\al \in V^{1,0}_2(Q_{T_{\rm{max}}})$, $\alpha= 1,2$.
\end{lemma}

\begin{proof}
1. If $u_{b\al}= 0$ or we have zero-flux boundary condition, then multiplying \eqref{PSPZERO} by $u_\al$ and integrating in
$\Omega$, we obtain
$$
  \frac{1}{2} \frac{d}{dt} \int_\Omega (u_\al)^2 \,d\mathbf{x} +
  \int_\Omega A^{jk}_{\al\bt}
  \, \frac{\partial u_\bt}{\partial x_k} \, \frac{\partial u_\al}{\partial
  x_j}\,d\mathbf{x} =
  \int_\Omega \vp^{j}_{\al} \, \frac{\partial u_\al}{\partial
  x_j}\,d\mathbf{x}
  + \int_\Omega g_{\al} \, u_\al \,d\mathbf{x}.
$$
Therefore, from \eqref{REG1} and $u_\al \in B_\Delta$,
there exists a positive constant $C_1$, such that
$$
  \frac{1}{2} \frac{d}{dt} \int_\Omega |u_\al|^2 \,d\mathbf{x} +
  \frac{\la_0}{2} \int_\Omega |\nablax u_\al|^2 \, d\mathbf{x} \leq C_1.
$$
Then, we have for each $t \in [0,T_{\rm{max}}]$
\begin{equation}
  \int_\Omega |u_\al(t)|^2 \,d\mathbf{x} +
  \la_0 \iint_{Q_{T_{\rm{max}}}} |\nablax u_\al(t)|^2 \, d\mathbf{x}
  \leq 2 C_1 + \int_\Omega |u_{0\al}|^2 \,d\mathbf{x} =:C_2,
\end{equation}
from which it follows that $u_\al \in
V^{1,0}_2(Q_{T_{\rm{max}}})$. Observe that $C_2$ does not depend
on  $t  < T_{\rm{max}}$.

\medskip
2. Now, if $u_{b\al} \neq 0$, then let us consider the system
$$
    \begin{aligned}
    \partial_t \mathbf{v}&= \divx (\nablax \mathbf{v}) \; &\text{in $Q_{T_{\rm{max}}}$},
    \\
   \mathbf{v}&= \mathbf{u}_{b} \; &\text{on $\Gamma_{T_{\rm{max}}}$},
    \\
    \mathbf{v}|_{t=0}&= \mathbf{v}_{0} \; &\text{in $\Omega$},
    \end{aligned}
$$
where $\mathbf{v}_{0}$ is a function in $C(\bar{\Omega})$,
satisfying the compatibility condition $$v_{0\al}(\mathbf{x})= u_{b\al}(0,\mathbf{x}),$$
for $\mathbf{x}  \in \partial \Omega$. It is clear that, $\mathbf{v} \in H_0^{2+\gamma}(Q_{T_{\rm{max}}})$, and
making $\mathbf{\tilde{u}}= \mathbf{u} - \mathbf{v}$, we apply the same argument as in item 1,
multiplying now by $\mathbf{\tilde{u}}$.
\end{proof}

\bigskip
Before we pass to consider separated the Dirichlet and flux conditions, let us consider another important restriction, which will be used in particular for 
$d \geq 2$ with the Dirichlet boundary condition, that is

\begin{equation}\label{COND2}
    \partial \varphi^{j}_{2} / \partial u_1 \equiv 0.
\end{equation}

\subsubsection{ Dirichlet condition } \label{SECDIRICHLET}

First, let us consider Dirichlet boundary condition in \eqref{BCAZERO}, i.e. for $\alpha =1, 2$. 
To begin, we work with equation \eqref{PSPZERO}, when $\alpha= 2$. 

\begin{lemma}
\label{LCONU2D} Let $0<\ve< T_{\rm{max}}$ be fixed and consider for $\alpha= 2$ the
initial-boundary value problem \eqref{PSPZERO}--\eqref{BCAZERO} (Dirichlet condition), with $u_{0_2} \in E$
and $u_{b_2} \in H_0^{2+\gamma}(\Gamma_T)$, $\gamma \in (0,1)$.
Assume the conditions 
\eqref{AIJ21}, \eqref{IBCBD}--\eqref{FCB}. Then,  there exists
$\gamma_1 \in (0,\gamma)$, such that
$$
  u_2 \in C^{2+\gamma_1,1+ \gamma_1/2}(Q^\ve_{T_{\rm{max}}}
  \cup \Gamma^\ve_{T_{\rm{max}}})
  \quad \text{and} \quad |u_2|_{\gamma_1}^{Q^\ve_{T_{\rm{max}}}} \leq
  C_2,
$$
where $C_2$ is independent of $t < T_{\rm{max}}$.
\end{lemma}

\begin{proof}
First, from Lemma \ref{MainLemma}, we have $u_2 \in V^{1,0}_2(Q_{T_{\rm{max}}})$.
The thesis follows applying Theorem 10.1 in Chapter III (linear theory) from
Ladyzenskaja, Solonikov and Ural'ceva \cite{OALVASNNU}, with
$$
  a_{jk}(t,\mathbf{x})=
  A_{22}^{jk}(\mathbf{x},u_1(t,\mathbf{x}),u_2(t,\mathbf{x})), \quad
  f_j(t,\mathbf{x})=
  \vp^j_2(\mathbf{x},u_1(t,\mathbf{x}),u_2(t,\mathbf{x})),
$$
also
$$
  f(t,\mathbf{x})=
  g_2(\mathbf{x},u_1(t,\mathbf{x}),u_2(t,\mathbf{x})),\quad a_j=
  b_j= a= 0.
$$
Since $\mathbf{u} \in B_\Delta$, and $f,f_j$ are uniformly
bounded, for each $d \geq 1$, there exist positive $q$ and $r$, such that
$$
  \frac{1}{r}+\frac{d}{2q}= 1 - \kappa_1, \quad \kappa_1 \in
  (0,1)
$$
and $\|f_j^2,f\|_{q,r,Q_{T_{\rm{max}}}}\leq C$, where $C$ is a
positive constant independent of $t < T_{\rm{max}}$.
\end{proof}

\begin{proposition}
\label{LCONGU21D}  Under conditions of Lemma \ref{LCONU2D}, and also condition \eqref{COND1}, 
then for $t < T_{\rm{max}}$ and $j= 1, \ldots, d$
$$
  \int_\ve^t\!\!\!\int_\Omega |\partial_{x_j} u_2(\tau,\mathbf{x})|^{4} \,
  d\mathbf{x}d\tau \leq C,
$$
where $C$ does not dependent on $t$.
\end{proposition}

\begin{proof}
1. First, we consider without loss of generality that, $u_{b2}= 0$. Otherwise, we proceed exactly as in the proof of Lemma
\ref{MainLemma}. For $\rho>0$ and any $\mathbf{x}_0 \in \Omega$, let $\Omega_{2\rho}= B_{2 \rho}(\mathbf{x_0}) \cap \Omega$.
 Let $0 \leq \zeta(t,\mathbf{x}) \leq 1$ be a smooth function,
such that, for each $t \in (\varepsilon, T_{\rm{max}})$, $0 < \varepsilon < T_{\rm{max}}$, 
$\zeta(t,\mathbf{x}) \equiv 0$ in  $(\varepsilon, T_{\rm{max}}) \times (\Omega \setminus B_{2\rho})\bigcup((0,\varepsilon) \times \Omega)$
and $\zeta(t,\mathbf{x}) \equiv 1$ in  $(\varepsilon, T_{\rm{max}}) \times \Omega_{\rho}$, with the obvious notation.

\medskip
2. For $\alpha= 2$, we multiply equation \eqref{PSPZERO} by
$$
    \frac{\partial}{\partial x_l} \big( \frac{\partial}{\partial x_l} u_2  \; \zeta^2 \big),
$$
and integrating in $(\varepsilon, t) \times \Omega_{2 \rho}$, it follows that
$$
    \begin{aligned}
    0= \int_\varepsilon^t \!\!\! \int_{\Omega_{2 \rho}} &\Big( - \partial_t u_2 +  \frac{\partial}{\partial x_j}
    \big( A_{22}^{jk} \; \frac{\partial u_2}{\partial x_k} - \varphi_2^j \big) + g_2 \Big)
    \frac{\partial}{\partial x_l} \big( \frac{\partial}{\partial x_l} u_2  \; \zeta^2 \big) \; d\mathbf{x} dt
    \\[5pt]
    &=  \int_\varepsilon^t \!\!\! \int_{\Omega_{2 \rho}} {u_2}_{tx_l} \; {u_2}_{x_l} \; \zeta^2  \; d\mathbf{x} dt
    \\[5pt]
    &+ \int_\varepsilon^t \!\!\! \int_{\Omega_{2 \rho}}  \frac{\partial}{\partial x_l}
    \big( A_{22}^{jk} \;  {u_2}_{ x_k}  \big) \; \frac{\partial}{\partial x_j}  ({u_2}_{x_l}  \; \zeta^2)  \; d\mathbf{x} dt
    \\[5pt]
    &- \int_\varepsilon^t \!\!\! \int_{\Omega_{2 \rho}}  \frac{\partial}{\partial x_l}
    \big( \varphi_2^j \big) \; \frac{\partial}{\partial x_j}  ({u_2}_{x_l}  \; \zeta^2)  \; d\mathbf{x} dt
    + \int_\varepsilon^t \!\!\! \int_{\Omega_{2 \rho}}  g_2  \; \frac{\partial}{\partial x_l}  ({u_2}_{x_l}  \; \zeta^2)  \; d\mathbf{x} dt
    \\[5pt]
    & =: I_1 + I_2 - I_3 + I_4.
    \end{aligned}
$$
Then, we have
\begin{equation}
\label{I1}
   I_1= \Big(  \frac{1}{2} \int_{\Omega_{2 \rho}}
   {u_2}^2_{x_l}  \zeta^2 \, d\mathbf{x} \Big)\Big|_\varepsilon^t
   -\int_\varepsilon^t \!\!\! \int_{\Omega_{2 \rho}} {u_2}^2_{x_l}  \zeta \zeta_t \;  d\mathbf{x} dt.
\end{equation}
Moreover, from $I_2$ we obtain
$$
\begin{aligned}
   I_2&= \int_\varepsilon^t \!\!\! \int_{\Omega_{2 \rho}}  A_{22}^{jk} \; \frac{\partial^2 u_2}{\partial x_k \partial x_l}
    \; \frac{\partial^2 u_2}{\partial x_l \partial x_j}  \zeta^2  d\mathbf{x} dt + \int_\varepsilon^t \!\!\! \int_{\Omega_{2 \rho}} 
    A_{22}^{jk} \; \frac{\partial^2 u_2}{\partial x_k \partial x_l}
    \;  {u_2}_{x_l} 2 \zeta {\zeta}_{x_j} d\mathbf{x} dt
    \\[5pt]
    & + \int_\varepsilon^t \!\!\! \int_{\Omega_{2 \rho}} \frac{\partial A_{22}^{jk}}{\partial u_2} \;  {u_2}_{x_l}
    \; {u_2}_{x_k} \;  \frac{\partial^2 u_2}{\partial x_l \partial x_j} \zeta^2 d\mathbf{x} dt
    \\[5pt]
    &+ \int_\varepsilon^t \!\!\! \int_{\Omega_{2 \rho}} \frac{\partial A_{22}^{jk}}{\partial u_2} \;  {u_2}_{x_l}
    \; {u_2}_{x_k} \;  {u_2}_{x_l} 2 \zeta {\zeta}_{x_j} d\mathbf{x} dt
    \\[5pt]
    & + \int_\varepsilon^t \!\!\! \int_{\Omega_{2 \rho}} 
    \frac{\partial A_{22}^{jk}}{\partial x_l} \;  {u_2}_{x_k} \;  \frac{\partial^2 u_2}{\partial x_l \partial x_j} \; \zeta^2
    + \int_\varepsilon^t \!\!\! \int_{\Omega_{2 \rho}} \frac{\partial A_{22}^{jk}}{\partial x_l} \;  
     {u_2}_{x_k} \;  {u_2}_{x_l} \; 2 \zeta {\zeta}_{x_j}  \, d\mathbf{x} dt
\\[5pt]
  &=: J_1 + K_1 + K_2 + K_3 + K_4 + K_5.
\end{aligned}
$$
Denoting by $D^2_{\mathbf{x}} \,  u_2$ the Hessian of the function $u_2$, we consider the following estimates:
$$
\begin{aligned}
    J_1 &\geq \lambda_0^2 \int_\varepsilon^t \!\!\! \int_{\Omega_{2 \rho}} |D^2_{\mathbf{x}} u_2|^2 \; \zeta^2 \, d\mathbf{x} dt,
\\[7pt]
   |K_1| &\leq \varepsilon_1 \int_\varepsilon^t \!\!\! \int_{\Omega_{2 \rho}} |D^2_{\mathbf{x}} u_2|^2 \; \zeta^2 \, d\mathbf{x} dt + C_1,
\\[7pt]
  |K_2| &\leq \varepsilon_2 \int_\varepsilon^t \!\!\! \int_{\Omega_{2 \rho}} |D^2_{\mathbf{x}} u_2|^2 \; \zeta^2 \, d\mathbf{x} dt
    + C_2  \int_\varepsilon^t \!\!\! \int_{\Omega_{2 \rho}} |\nablax u_2|^4 \, d\mathbf{x} dt ,
\\[7pt]
    |K_3| &\leq  \int_\varepsilon^t \!\!\! \int_{\Omega_{2 \rho}} |\nablax u_2|^4 \; \zeta^2 \, d\mathbf{x} dt  + C_3 ,
\\[7pt]
    |K_4| &\leq \varepsilon_4 \int_\varepsilon^t \!\!\! \int_{\Omega_{2 \rho}} |D^2_{\mathbf{x}} u_2|^2 \; \zeta^2 \, d\mathbf{x} dt  + C_4,
\quad
  |K_5|  \leq C_5,
  \end{aligned}
$$
where we have used Lemma \ref{MainLemma}, the condition \eqref{REG1} and the generalized Young's inequality. Therefore, from the above
estimates
\begin{equation}
  \begin{aligned}
       I_1 + J_1 &\leq  |I_3| + |I_4| + \sum_{\iota= 1}^5 |K_\iota| \leq  |I_3| + |I_4| 
\\
                          &+  \varepsilon_6 \int_\varepsilon^t \!\!\! \int_{\Omega_{2 \rho}} |D^2_\mathbf{x} u_2|^2 \; \zeta^2 \, d\mathbf{x} dt
                          +  C_6  \int_\varepsilon^t \!\!\! \int_{\Omega_{2 \rho}} |\nablax u_2|^4  \; \zeta^2 \, d\mathbf{x} dt + C_7.
      \end{aligned}
\end{equation}
Now, we proceed to estimate $I_3$, first
$$
    \begin{aligned}
    I_3 =  \int_\varepsilon^t \!\!\! \int_{\Omega_{2 \rho}} &\Big(\frac{\partial \varphi_2^j}{\partial u_1} \; {u_1}_{x_l}
    \; \frac{\partial^2 u_2}{\partial x_l \partial x_j} \; \zeta^2
    + \frac{\partial \varphi_2^j}{\partial u_1} \; {u_1}_{x_l}
    \; {u_2}_{x_l} \; 2 \zeta \zeta_{x_j}
    \\[5pt]
    &+ \frac{\partial \varphi_2^j}{\partial u_2} \; {u_2}_{x_l}
    \; \frac{\partial^2 u_2}{\partial x_l \partial x_j} \; \zeta^2
    + \frac{\partial \varphi_2^j}{\partial u_2} \; {u_2}_{x_l}
    \; {u_2}_{x_l} \; 2 \zeta \zeta_{x_j}
    \\[5pt]
    &+ \frac{\partial \varphi_2^j}{\partial x_l}
    \; \frac{\partial^2 u_2}{\partial x_l \partial x_j} \; \zeta^2
    + \frac{\partial \varphi_2^j}{\partial x_l}
    \; {u_2}_{x_l} \; 2 \zeta \zeta_{x_j})
     \Big)\, d\mathbf{x} dt,  
      \end{aligned}
$$
then we consider the following estimates:
$$
  \begin{aligned}
   &\int_\varepsilon^t \!\!\! \int_{\Omega_{2 \rho}}  \frac{\partial \varphi_2^j}{\partial u_1} \; {u_1}_{x_l}
    \; \frac{\partial^2 u_2}{\partial x_l \partial x_j} \; \zeta^2  \, d\mathbf{x} dt
    \leq \varepsilon_8 \int_\varepsilon^t \!\!\! \int_{\Omega_{2 \rho}} |D^2_{\mathbf{x}} u_2|^2 \; \zeta^2 \, d\mathbf{x} dt  + C_8,
\\[7pt]
   &\int_\varepsilon^t \!\!\! \int_{\Omega_{2 \rho}}  \frac{\partial \varphi_2^j}{\partial u_2} \; {u_2}_{x_l}
    \; \frac{\partial^2 u_2}{\partial x_l \partial x_j} \; \zeta^2  \, d\mathbf{x} dt
    \leq \varepsilon_9 \int_\varepsilon^t \!\!\! \int_{\Omega_{2 \rho}} |D^2_{\mathbf{x}} u_2|^2 \; \zeta^2 \, d\mathbf{x} dt  + C_9,
\\[7pt]
   &\int_\varepsilon^t \!\!\! \int_{\Omega_{2 \rho}}  \frac{\partial \varphi_2^j}{\partial x_l}
    \; \frac{\partial^2 u_2}{\partial x_l \partial x_j} \; \zeta^2  \, d\mathbf{x} dt
    \leq \varepsilon_{10} \int_\varepsilon^t \!\!\! \int_{\Omega_{2 \rho}} |D^2_{\mathbf{x}} u_2|^2 \; \zeta^2 \, d\mathbf{x} dt  + C_{10},
\\[7pt]
   &\int_\varepsilon^t \!\!\! \int_{\Omega_{2 \rho}} \big( \frac{\partial \varphi_2^j}{\partial u_1} \; {u_1}_{x_l} +
   \frac{\partial \varphi_2^j}{\partial u_2} \; {u_2}_{x_l} \big)
    \; {u_2}_{x_l} \; 2 \zeta \zeta_{x_j}  \, d\mathbf{x} dt
    \leq  C_{11}.
\end{aligned}
$$
Hence we conclude that
\begin{equation}
\label{I3}
   |I_3| \leq \varepsilon_{12} \int_\varepsilon^t \!\!\! \int_{\Omega_{2 \rho}} |D^2_{\mathbf{x}} u_2|^2 \; \zeta^2 \, d\mathbf{x} dt  + C_{12}.
\end{equation}
Finally, we easily have
\begin{equation}
\label{I4} |I_4| \leq C_{13}.
\end{equation}
Taking  $\varepsilon_6 + \varepsilon_{12} \leq \lambda_0^2 / 4$, we obtain from \eqref{I1}--\eqref{I4}
\begin{equation}
\label{FINALEST}
\begin{aligned}
    \Big(  \frac{1}{2} \int_{\Omega_{2 \rho}}
   {u_2}^2_{x_l}  \zeta^2 \, d\mathbf{x} \Big)\Big|_\varepsilon^t
   &+ \frac{\lambda_0^2}{4} \int_\varepsilon^t \!\!\! \int_{\Omega_{2 \rho}} |D^2_{\mathbf{x}} u_2|^2 \; \zeta^2 \, d\mathbf{x} dt
   \\[5pt]
   &\leq C_{14} \int_\varepsilon^t \!\!\! \int_{\Omega_{2 \rho}} |\nablax u_2|^4 \; \zeta^2 \, d\mathbf{x} dt + C_{15}.
   \end{aligned}
\end{equation}

\medskip
3. Now, let us consider  Lemma 5.4 in Chapter II, from Ladyzenskaja, Solonikov and Ural'ceva
\cite{OALVASNNU}. Taking $s= 1$, we have
\begin{equation}
\label{LUSLEMMA54}
\begin{aligned}
   \int_{\Omega_{2 \rho}}  |\nablax u_2|^4 \; \zeta^2 \, d\mathbf{x}   \leq 16 \; {\rm osc}^2[u_2,\Omega_{2 \rho}] 
   \int_{\Omega_{2 \rho}}  C (|D^2_{\mathbf{x}} u_2|^2 \; \zeta^2 \,+  |\nablax u_2|^2 ) d\mathbf{x}.
\end{aligned}
\end{equation}
We recall that, $\rm{osc}[u(\mathbf{x});\Omega]$ is the oscillation of $u(\mathbf{x})$ in $\Omega$, which means the difference 
between $\rm{ess}\sup_{\Omega} u(\mathbf{x})$ and $\rm{ess}\inf_{\Omega} u(\mathbf{x})$, therefore, 
it follows from Lemma 3.6 that,  
$$
    16 \; \rm{osc}^2[u_2,\Omega_{2 \rho}]\leq C_{16} \; \rho^{2 \gamma_1}.
$$ 
Then, for $\rho \leq \rho_0$, such that, 
$\ \rho^{2 \gamma_1}_0 \, C_{16} \, C_{14} \, C \leq \lambda_0^2/8$, we obtain from \eqref{FINALEST}, \eqref{LUSLEMMA54}
$$
\Big(   \int_{\Omega_{2 \rho}}
  {u_2}^2_{x_l}  \zeta^2 \, d\mathbf{x} \Big)\Big|_\varepsilon^t
   +  \int_\varepsilon^t \!\!\! \int_{\Omega_{2 \rho}} |D^2_{\mathbf{x}} u_2|^2 \; \zeta^2 \, d\mathbf{x} dt
   \leq C_{17}  
$$
and 
$$
    \max_{t_1 \in (\varepsilon,t)} \int_{\Omega_{2 \rho}}
   {u_2}^2_{x_l}  \zeta^2 \, d\mathbf{x} (t_1)
  +  \int_\varepsilon^t \!\!\! \int_{\Omega_{2 \rho}} |D^2_{\mathbf{x}} u_2|^2 \; \zeta^2 \, d\mathbf{x} dt    
  +\int_\varepsilon^t \!\!\! \int_{\Omega_{2 \rho}} |\nablax u_2|^4 \; \zeta^2 \, d\mathbf{x} dt \leq  C_{18}.
$$
From this estimate and equation \eqref{PSPZERO}  $(\alpha= 2)$, we obtain 
$$
     \int_\varepsilon^t\!\!\!\int_{\Omega_{2 \rho}}(u_{2t})^2 \zeta^2 d\mathbf{x} dt \leq  C_{19}.
$$
Consequently, it follows from the above estimates 
\begin{equation}
\label{FINAL}
\begin{aligned}
  \max_{t_1 \in(\varepsilon,t)} \int_{\Omega_\rho}
   {u_2}^2_{x_l}  \zeta^2 \, d\mathbf{x} (t_1)
   &+  \int_\varepsilon^t \!\!\! \int_{\Omega_\rho} |D^2_{\mathbf{x}} u_2|^2 \; \zeta^2 \, d\mathbf{x} dt
   \\[5pt]
     &+\int_\varepsilon^t \!\!\! \int_{\Omega_\rho} |\nablax u_2|^4 \; \zeta^2 \, d\mathbf{x} dt 
     + \int_\varepsilon^t\!\!\!\int_{\Omega_{2 \rho}}(u_{2t})^2 \zeta^2 d\mathbf{x} dt \leq  C_{20}.
\end{aligned}
\end{equation} 

\medskip
4. Finally, the thesis of the Lemma is proved from estimate \eqref{FINAL} and a
standard argument of partition of unity subordinated to a finite local cover of $\Omega$, since
$\bar{\Omega}$ is a compact subset of $\R^d$.  
\end{proof}

\begin{lemma}
\label{LCONGU22D} For $d \geq 2$, $j=1, \ldots,d$, under conditions 
of Lemma \ref{LCONU2D}, and also conditions \eqref{COND1},
\eqref{COND2},
then for $t < T_{\rm{max}}$
$$
  \int_\ve^t\!\!\!\int_\Omega |\partial_{x_j} u_2(\tau,\mathbf{x})|^{2s+4} \,
  d\mathbf{x}d\tau \leq C,
$$
where $C$ does not dependent on $t$, and $s$ is any
non-negative integer, such that, $s > (d-2) / 2$.
\end{lemma}

\begin{proof}

Let us consider the second equation of our system, which is to say
$$
  \begin{aligned}
  \partial_t u_2&= \partial_{x_j} \Big(A^{jk}_{22}(\mathbf{x},u_2) \;
  \partial_{x_k}u_2 - \vp_2^j(\mathbf{x},u_2) \Big) + g_2(\mathbf{x},u_1,u_2),
\\
  u_2|_\Gamma&= u_{b_2},
\\
  u_2(0)&= u_{0_2}.
  \end{aligned}
$$
 This problem satisfies the conditions (3.1)-(3.6) in Section 3 of
Chapter V (non-linear theory) from Ladyzenskaja, Solonikov and Ural'ceva
\cite{OALVASNNU}. 
In this book is proved in Section 4 of Chapter V,  the estimates (4.10) for any $s>0$: 
$$
    \max_{\varepsilon\leq t\leq T}\int_{\Omega}|u_x|^{2s+2}dx+\int_0^T \int_{\Omega}|u_x|^{2s+4}dxdt\leq Const.
$$
The thesis of this lemma follows from the above estimate.
\end{proof}

\bigskip
Now, from the above estimates obtained for $u_2$, 
we are going to consider equation \eqref{PSPZERO}, when $\alpha= 1$.  Then, similarly to Lemma 
\ref{LCONU2D} we have the following

\begin{lemma}
Let $0<\ve< T_{\rm{max}}$ be fixed and consider for $\alpha= 1$ the
Dirichlet problem in \eqref{PSPZERO}--\eqref{BCAZERO}, with $u_{0_1} \in E$
and $u_{b_1} \in H_0^{2+\gamma}(\Gamma_T)$, $\gamma \in (0,1)$.
Assume the conditions \eqref{COND1},
\eqref{COND2} for $d \geq 2$,
\eqref{AIJ21}, \eqref{IBCBD}--\eqref{FCB}. Then, there exists $\gamma_2 >0$, such
that
\begin{equation}
\label{ESTUD1}
  |u_1|_{\gamma_2}^{Q_t^\ve} \leq C,
\end{equation}
where $C$ is a positive constant independent of $t$, with $t \in
(\ve,T_{\rm{max}})$.
\end{lemma}

\begin{proof}
Again the result follows applying Theorem 10.1 in Chapter III from
Ladyzenskaja, Solonikov and Ural'ceva \cite{OALVASNNU}, with
$$
  \begin{aligned}
  a_{jk}(t,\mathbf{x})&=
  A_{11}^{jk}(\mathbf{x},u_1(t,\mathbf{x}),u_2(t,\mathbf{x})), \\
  f_j(t,\mathbf{x})&= A_{12}^{jk}(\mathbf{x},u_1(t,\mathbf{x}),u_2(t,\mathbf{x}))
  \; \partial_{x_k} u_2 +
  \vp^j_1(\mathbf{x},u_1(t,\mathbf{x}),u_2(t,\mathbf{x})),
  \end{aligned}
$$
and
$$
  f(t,\mathbf{x})=
  g_1(\mathbf{x},u_1(t,\mathbf{x}),u_2(t,\mathbf{x})),\quad a_j=
  b_j= a= 0.
$$
Since $\mathbf{u} \in B_\Delta$, and considering the
higher estimates of $\partial_{x_k} u_2$ obtained from
Lemma \ref{LCONGU21D}  (in particular for $d= 1$), Lemma \ref{LCONGU22D} ($d> 1$),
we may consider $f,f_j^2$ uniformly bounded, i.e. 
$\|f_j^2,f\|_{q,r,Q_t^\varepsilon}   \leq C$, 
with $q=s+2$, $r=s+2$,
where $C$ is a
positive constant independent of $t < T_{\rm{max}}$.
Moreover, we have
$$
  \frac{1}{r}+\frac{d}{2q}= 1 - \kappa_1 \in (0,1),
$$
where
$$
  \kappa_1= 1- \frac{1}{s+2} - \frac{d}{2s+4} > 0
$$
and therefore we obtain \eqref{ESTUD1}.
\end{proof}

\begin{theorem} \label{DIRICHLETTHM}
Let any $T>0$ be given and consider $\mathbf{u}_{0} \in \! E$,
$\mathbf{u}_{b} \in \! H_0^{2+\gamma}(\Gamma_T)$, with $\gamma \in (0,1)$.
Assume the conditions \eqref{COND1},
\eqref{AIJ21}, \eqref{IBCBD}--\eqref{FCB}, \eqref{COND2} for $d \geq 2$. Then, 
the initial-boundary value problem \eqref{PSPZERO}--\eqref{BCAZERO},
with Dirichlet condition, has a unique global solution
$\mathbf{u} \in H_0^{2+\gamma}(\ol{Q}_T)$. Moreover, for each $(t,\mathbf{x}) \in \ol{Q}_T$, 
$\mathbf{u}(t,\mathbf{x}) \in B_\Delta$.
\end{theorem}

\begin{proof}
First, we have from the previous results that,
$u_\al \in C^\gamma(\ol{Q}_t^\ve)$ $(\al=1,2)$ for some
$\gamma>0$. Therefore, applying Corollary \ref{CORGS} we obtain the
global classical solution $\mathbf{u}$ for the 
initial-boundary value problem \eqref{PSPZERO}--\eqref{BCAZERO},
with Dirichlet condition. Finally, $\mathbf{u}(t,\mathbf{x}) \in B_\Delta$, for each $(t,\mathbf{x}) \in \ol{Q}_T$, 
follows from Corollary  \ref{CORIR}.

\end{proof}

\subsubsection{ Flux condition } \label{SECFLUX}

Now, let us consider flux-boundary condition in \eqref{BCAZERO}, for $\alpha =1, 2$.
Analogously, we begin establishing estimates for $u_2$. Then, we have the following 

\begin{lemma}
\label{LCONU2F} Let $0<\ve< T_{\rm{max}}$ be fixed and consider for $\alpha= 2$ the
initial-boundary value problem \eqref{PSPZERO}--\eqref{BCAZERO} (flux condition), with $u_{0_2} \in E$.
Assume the conditions \eqref{COND1}, \eqref{COND2},
\eqref{AIJ21}, \eqref{IBCBD}--\eqref{FCB}. Then,  there exists
$\gamma_1 > 0$, such that
$$
  u_2 \in C^{2+\gamma_1,1+ \gamma_1/2}(Q^\ve_{T_{\rm{max}}}
  \cup \Gamma^\ve_{T_{\rm{max}}})
  \quad \text{and} \quad |u_2|_{\gamma_1}^{Q^\ve_{T_{\rm{max}}}} \leq
  C_2,
$$
where $C_2$ is independent of $t < T_{\rm{max}}$.
\end{lemma}

\begin{proof}
1. First, the interior estimates in $C^{2+\gamma_1,1+ \gamma_1/2}({Q^\ve}'_{T_{\rm{max}}})$, 
$$\text{for any $ {Q^\ve}'_{T_{\rm{max}}} \subset \subset  {Q^\ve}_{T_{\rm{max}}}$},$$
follow applying Theorem 10.1 in Chapter III  (linear theory)  from
Ladyzenskaja, Solonikov and Ural'ceva \cite{OALVASNNU}, with
$$
  a_{jk}(t,\mathbf{x})=
  A_{22}^{jk}(\mathbf{x}, u_2(t,\mathbf{x})), \quad
  f_j(t,\mathbf{x})=
  \vp^j_2(\mathbf{x}, u_2(t,\mathbf{x})),
$$
also
$$
  f(t,\mathbf{x})=
  g_2(\mathbf{x},u_1(t,\mathbf{x}),u_2(t,\mathbf{x})),\quad a_j=
  b_j= a= 0.
$$
Since $\mathbf{u} \in B_\Delta$, and $f,f_j$ are uniformly
bounded, for each $d \geq 1$, there exist positive $q$ and $r$, such that
$$
  \frac{1}{r}+\frac{d}{2q}= 1 - \kappa_1, \quad \kappa_1 \in
  (0,1)
$$
and $\|f_j^2,f\|_{q,r,Q_{T_{\rm{max}}}}\leq C$, where $C$ is a
positive constant independent of $t < T_{\rm{max}}$.

\bigskip
2. Now, to derive the estimates closely to the boundary, we may apply the
non-linear theory from
Ladyzenskaja, Solonikov and Ural'ceva \cite{OALVASNNU} developed in Chapter V, Section 7.
More precisely, the equation for $u_2$ is in the same form of (7.1)-(7.3), and admits the conditions 
(7.4)-(7.6). Therefore, it follows
from Theorem 7.1 in that book,  the thesis of the lemma.
\end{proof}

\begin{lemma}
\label{LCONU2F1}  Under conditions of Lemma \ref{LCONU2F}, there exists $M_1> 0$, such that
$$
  \sup_{Q^\ve_{T_{\rm{max}}}} |\nabla_{\mathbf{x}} u_2(t,\mathbf{x})| \leq M_1.
$$
\end{lemma}

\begin{proof}
The thesis of the Lemma follows directly from Theorem 3 in the paper of 
A.I. Nazarov, N.N. Uraltseva \cite{AINNNU}. Also, from Theorem 4 in that paper, 
there exists $M_{1+\gamma}> 0$, such that
$$
    \|  u_2(t,\mathbf{x})\|_{1+\gamma}^{Q^\ve_{T_{\rm{max}}}} \leq M_{1+\gamma}.
$$
\end{proof}

\medskip
Now, let us rewrite the second equation of our system, which is to say
$$
   \begin{aligned}
    \partial_t u_2 + \partial_{x_j} \vp_2^j &+ \partial_{u_2} \vp_2^j \partial_{x_j} u_2
    = A_{22}^{kj}(\mathbf{x}, u_2) \partial^2_{x_k x_j} u_2 
    \\
    &+ \partial_{u_2} \big(A_{22}^{kj}(\mathbf{x}, u_2)\big) \partial_{x_k} u_2 \; \partial_{x_j} u_2
    + \partial_{x_j}\big( A_{22}^{kj}(\mathbf{x}, u_2)\big) \partial_{x_k} u_2
    \\
    & + g_2(\mathbf{x},u_1,u_2), \qquad (t,\mathbf{x}) \in {Q^\ve_{T_{\rm{max}}}}.
    \end{aligned}
$$
Supplemented with the boundary condition 
$$
   \Big( A_{22}^{kj}(\mathbf{x}, u_2) \partial_{x_k} u_2 + \vp_2^{j}(\mathbf{x}, u_2) \Big) \cos(\mathbf{n},x_j)= 0,
   \qquad (t,\mathbf{x}) \in {\Gamma^\ve_{T_{\rm{max}}}}.
$$
Also we have $u_2(\ve,\cdot) \in C^{2+\gamma}(\ol{\Omega})$. From the above estimates, Lemma \ref{LCONU2F} and 
Lemma \ref{LCONU2F1}, we obtain 

$$
   | \partial_{u_2} \big(A_{22}^{kj}\big) \partial_{x_k} u_2 \; \partial_{x_j} u_2
    + \partial_{x_j}\big( A_{22}^{kj}\big) \partial_{x_k} u_2 + g_2 -  \partial_{x_j} \vp_2^j 
    - \partial_{u_2} \vp_2^j \partial_{x_j} u_2 |_0^ {Q^\ve_{T_{\rm{max}}}}    \leq M_2,
$$
and 
$$
  | A_{22}(\mathbf{x},u_2) ^\gamma|_0^ {Q^\ve_{T_{\rm{max}}}}    \leq M_3.
$$
Moreover, it is possibly to apply the $W^{1,2}_p$--estimates  for linear parabolic problem
with flux boundary conditions (V.A. Solonikov \cite{VAS}, Amann \cite{HA0} )
$$
   \|u_2\|_{W^{1,2}_p(Q^\ve_{T_{\rm{max}}})} \leq M_4
$$
for any $p >1$, where the positive constant $M_4$ depend on $p$.

\begin{lemma} 
Let $0<\ve< T_{\rm{max}}$ be fixed and consider for $\alpha= 1$ the
initial-boundary value problem \eqref{PSPZERO}--\eqref{BCAZERO} (flux condition), with $u_{0_1} \in E$.
Assume the conditions \eqref{COND1}, \eqref{COND2},
\eqref{AIJ21}, \eqref{IBCBD}--\eqref{FCB}. Then,  there exists $\gamma_2 >0$, such that
\begin{equation}
\label{ESTU1}
  |u_1|_{\gamma_2}^{Q_t^\ve} \leq C,
\end{equation}
where $C$ is a positive constant independent of $t$, with $t \in
(\ve,T_{\rm{max}})$. 
\end{lemma}

\begin{proof}
1. Again, the interior estimates  in $C^{\gamma_2}({Q^\ve}'_{T_{\rm{max}}})$, 
$( {Q^\ve}'_{T_{\rm{max}}} \subset \subset  {Q^\ve}_{T_{\rm{max}}})$, follow applying Theorem 10.1 (linear theory) in Chapter III from
Ladyzenskaja, Solonikov and Ural'ceva \cite{OALVASNNU}, with
$$
  \begin{aligned}
  a_{jk}(t,\mathbf{x})&=
  A_{11}^{jk}(\mathbf{x},u_1(t,\mathbf{x}),u_2(t,\mathbf{x})), \\[5pt]
  f_j(t,\mathbf{x})&=
  \vp^j_1(\mathbf{x},u_1(t,\mathbf{x}),u_2(t,\mathbf{x}))
  - A_{12}^{jk}(\mathbf{x},u_1(t,\mathbf{x}),u_2(t,\mathbf{x})) \, \partial_{x_k} u_2(t,\mathbf{x}),
 \end{aligned}
$$
also
$$
  f(t,\mathbf{x})=
  g_2(\mathbf{x},u_1(t,\mathbf{x}),u_2(t,\mathbf{x})),\quad a_j=
  b_j= a= 0.
$$
Since $\mathbf{u} \in B_\Delta$, and $f,f_j$ are uniformly
bounded, for each $d$, there exist positive $q$ and $r$, such that
$$
  \frac{1}{r}+\frac{d}{2q}= 1 - \kappa_1, \quad \kappa_1 \in
  (0,1)
$$
and $\|f_j^2,f\|_{q,r,Q_{T_{\rm{max}}}}\leq C$, where $C$ is a
positive constant independent of $t < T_{\rm{max}}$.

\bigskip
2. Let us observe that, the estimate closely 
to the boundary  does not follow in the same way to $u_1$, that is to say, it is not possible to apply the same strategy as done 
for $u_2$ in Lemma \ref{LCONU2F}. Therefore, we proceed to derive the estimates closely to the boundary, applying the
non-linear theory from
Ladyzenskaja, Solonikov and Ural'ceva \cite{OALVASNNU} developed in Section 7 of Chapter V.
More precisely, the equation for $u_1$ is in the same form of (7.1)-(7.3), and admits the conditions 
(7.4)-(7.6) from that book.  Then, from Theorem 7.1 in that book, it follows the thesis of the lemma.
\end{proof}

\begin{theorem} \label{FLUXTHM}
Let any $T>0$ be given and consider $\mathbf{u}_{0} \in \!E$.
Assume the conditions \eqref{COND1},
\eqref{COND2},
\eqref{AIJ21}, \eqref{IBCBD}--\eqref{FCB}. Then, 
the initial-boundary value problem \eqref{PSPZERO}--\eqref{BCAZERO},
with flux condition, has a unique global solution
$\mathbf{u} \in H_0^{2+\gamma}(\ol{Q}_T)$. Moreover, for each $(t,\mathbf{x}) \in \ol{Q}_T$, 
$\mathbf{u}(t,\mathbf{x}) \in B_\Delta$.
\end{theorem}

\begin{proof}
First, we have from the above lemmas that,
$u_\al \in C^\gamma(\ol{Q}_t^\ve)$ $(\al=1,2)$ for some
$\gamma>0$. Therefore, from Corollary \ref{CORGS} we obtain the
global classical solution $\mathbf{u}$ for the problem
\eqref{PS}--\eqref{BC}.
\end{proof}

\begin{remark}
Clearly, it remains to consider Dirichlet boundary condition for $\alpha= 1$, flux-condition for $\alpha= 2$,
and vice-versa. These type of mixed boundary conditions follow from suitable consideration of the above
results.  
\end{remark}

\section*{Acknowledgements}

The authors were partially supported by FAPERJ through the grant
E-26/ 111.564/2011 entitled \textsl{"Fractured Porous Media"}.
Wladimir Neves is also partially supported by Pronex-FAPERJ
through the grant E-26/ 110.560/2010 entitled \textsl{"Nonlinear Partial
Differential Equations"}.



\begin{thebibliography}{99}

\bibitem{TAAMPV1}  {\sc T.A.~Akramov, and M.P.~Vishnevskii}, {\em Qualitative properties for reaction-diffusion systems}, Siberian
Math. Journal, {\bf 36}, vol.1, (1995), 1--16.

\bibitem{TAAMPV2} {\sc T.A.~Akramov, and M.P.~Vishnevskii},
{\em Global solvability of a reaction-diffusion system}, Math.
Modelling, {\bf 4} , vol.11, (1992), 110--120, (in russian).


\bibitem{HA0}{\sc H.~Amann}, {\em Linear and quasilinear 
parabolic problems}, Birkh\"{a}user, Basel, 1995.

\bibitem{HA1}{\sc H.~Amann}, {\em Dynamic theory of quasilinear
parabolic systems III. Global existence}, Math. Z., {\bf 202}, (1989),
219--250.

\bibitem{HA2}{\sc H.~Amann}, {\em Dynamic theory of quasi linear
parabolic equations II. Reaction-diffusion systems}, 
Differential and Integral Equations, {\bf 3}, (1990), 13--75.

\bibitem{AA}{\sc A.~Arkhipova}, {\em On classical solvability of
the Cauchy-Dirichlet problem for nondiagonal parabolic systems in
the case of two spatial variables}, Amer. Math. Soc. Translations,
Serie 2, vol. 209, (2003), 1--19.

\bibitem{VSB1} {\sc V.S.~Belonosov}, {\em Estimates of solutions of
parabolic systems in H\"{o}lder weight classes and some of their
applications}, Mat. Sb., {\bf 110}, n.2, (1979), 163--188.

\bibitem{VSB2} {\sc V.S.~Belonosov}, {\em Estimates of solutions of
parabolic systems in H\"{o}lder classes with weight}, Dokl. Akad.
Nauk SSSR,  {\bf 241}, n.2, (1978), 265--268.

\bibitem{VSBTIZ}{\sc V.S.Belonosov, and T.I.Zelenjak}, {\em Nonlocal
problems in theory of quasilinear parabolic equations}, Novosibirsk,
Izd. Novosibirsk University, 1975, (in Russian).

\bibitem{SBRBHF}{\sc S. Berres, R. B\"{u}rger, and H. Frid},
{\em Neumann problems for quasi-linear parabolic systems modeling
polydisperse suspensions} SIAM J. Math. Anal., {\bf 38}, (2006),
557--573.

\bibitem{CCS}{\sc K. Chueh, C. Conley, and J. Smoller},
{\em Positively invariant regions for systems of nonlinear
difusion equations.} Ind. U. Math. J., {\bf 26} (1977), 373--392.

\bibitem{TCPL}{\sc T. Cie\'slak, and P. Laurençot},
{\em Finite time blow-up for a one-dimensional quasilinear
parabolic-parabolic chemotaxis system}, Annales de L'Institute
Henri Poincar\'{e} (C), Analyse Non-lin\'{e}aire, {\bf 27}, (2010), 437--446.

\bibitem{ACJEZY}{\sc A.~Constantin, J.~Escher, and Z.Yin}, {\em Global
solutions for quasilinear parabolic systems}, in J. Diff. Eq., 
{\bf 197}, (2004), 73--84.

\bibitem{HFVC}{\sc H.~Frid, and V.~Shelukhin}, {\em Initial boundary value
problems for a quasilinear parabolic system in three-phase
capillary flow in porous media}, SIAM J. Math. Anal., {\bf 36}, (2005),
1407--1425.

\bibitem{MGGM}{\sc M.~Giaquinta, and G.~Modica}, {\em Local existence for
quasilinear parabolic systems under nonlinear boundary conditions},
Annali di Mattematica Pura ed Applicata, {\bf 149}, (1987), 41--59.

\bibitem{MGMS}{\sc M.~Giaquinta, and M.~Struwe}, {\em An optimal regularity
result for a class of quasilinear parabolic systems}, Manuscripta
Math., {\bf 36}, (1981), 223--239.

\bibitem{JSJO2}{\sc O.~John, and J.~Star\'{a}}, {\em Some (new) counterexamples
of parabolic systems}, Comment Math. Univ. Carolin., {\bf 36}, n.3,
(1995), 503--510.

\bibitem{JSJO1}{\sc O.~John, J.~Mal\'{y}, and J.~Star\'{a}}, {\em Counterexamples
to the regularity of weak solution of the quasilinear parabolic
systems}, Comment Math. Univ. Carolin., {\bf 27}, n.1, (1986),
123--126.

\bibitem{EK}{\sc E. Kalita},
{\em On the H\"{o}lder continuity of solutions of nonlinear parabolic
systems}, Comment. Math. Univ. Carolin., {\bf 35}, n.4, (1994), 675--680.

\bibitem{OALVASNNU}{\sc O.A.~Ladyzenskaja, V.A.~Solonnikov, and
N.N.~Ural´ceva}, {\em Linear and Quasi-linear Equations of
Parabolic Type}, American Mathematical Society, Mathematical
Monographs, vol. 23, 1988.

\bibitem{AINNNU} {\sc A.I.,~ Nazarov, and  N.N.~Ural´ceva}, 
 {\em A problem with an obligue derivative for a 
quasilinear parabolic equation}, J. Math. Sci., {\bf 77}, (1995), 1587--1593.

\bibitem{DS1}{\sc D. Serre}, {\em Systems of conservation laws II.
Geometric structures, oscillations, and initial-boundary value
problems}, Cambridge University Press, Cambridge, 2000. Translated
from the 1996 French original by I. N. Sneddon.

\bibitem{JS}{\sc J. Smoller}, {\em Shock Waves and Reaction-Diffusion Equations},
Springer-Verlag, New York, Berlin, Heidelberg, 2nd ed., 1994.

\bibitem{VAS}{\sc V.A.~Solonnikov}, {\em On boundary value problems for linear 
parabolic systems of differential equations of general form}, Proc. Math. Inst. Steklov,
{\bf 83}, (1966), 1--184.

\bibitem{VASAGK}{\sc V.A.~Solonnikov and A.G.~Khachatryan},
{\em Estimates for solutions of parabolic initial-boundary value
problems in weighted H\"{o}lder norms}, Tr. Mat. Inst. Steklova,
{\bf 147}, (1980), 147--155.

\bibitem{[16]}{\sc M.~Weigler}, {\em Global solution to a class of
strongly coupled parabolic system}, Math. Anal., {\bf 292}, (1992),
711--727.

\end{thebibliography}
\end{document}